\documentclass{svproc}
\usepackage{url}

\usepackage{amsmath,amsfonts,amssymb,url}
\usepackage{color}
\usepackage{algorithm,graphicx}
\usepackage[noend]{algpseudocode}
\usepackage{natbib}

\newcommand{\norm}[1]{\left\lVert#1\right\rVert}

\newcommand{\innerproduct}[2]{\left\langle #1, #2 \right\rangle}

\DeclareMathOperator*{\argmin}{arg\,min}

\begin{document}
\mainmatter              
\title{Local linear smoothing in additive models as data projection}
%
%
\author{Munir Hiabu\inst{1} \and Enno Mammen\inst{2} \and
Joseph T. Meyer\inst{2} }
%
%
%
\institute{University of Copenhagen, Department of Mathematical Sciences, Universitetsparken 5, 2100 Copenhagen O,
Denmark\\
\and
Heidelberg University,
Institute for Applied Mathematics,  INF 205,\\
69120 Heidelberg, Germany\\
}

\maketitle              

\begin{abstract}
We discuss local linear smooth backfitting for additive nonparametric models. This procedure is well known for achieving optimal convergence rates under appropriate smoothness conditions. In particular, it allows for the estimation of each component of an additive model with the same asymptotic accuracy as if the other components were known. The asymptotic discussion of local linear smooth backfitting is rather complex because typically an overwhelming notation is required for a detailed discussion. In this paper we interpret the  local linear smooth backfitting estimator as a projection of the data onto a linear space with a suitably chosen semi-norm. This approach simplifies both the mathematical discussion as well as the intuitive understanding of properties of this version of smooth backfitting. 	

\end{abstract}

	\section{Introduction}
	In this paper we consider local linear smoothing in an additive model 
	\begin{align}\label{model:add}
	E[Y_i | X_i ] =  m_0+m_1(X_{i1})+\cdots +m_d(X_{id}),
	\end{align}
	where $(Y_i,X_i)$ $(i=1,\dots,n)$ are $iid$ observations with values in $\mathbb R \times \mathcal X$ for a bounded connected open subset $\mathcal X\subseteq\mathbb R^d$.
	Here, $m_j$ $(j=1,\dots,d)$ are some smooth functions which we aim to estimate and $m_0\in\mathbb{R}$. Below, we will add norming conditions on $m_0,\dots,m_d$ such that they are uniquely defined given the sum. 
	In \cite{mammen1999existence} a local linear  smooth backfitting estimator based on smoothing kernels was proposed for the additive functions $m_j$. There, it was shown that their version of a  
	local linear estimator $\hat m_j$ of the function $m_j$ has the same pointwise asymptotic variance and bias as a classical local linear estimator in the oracle model, where one observes i.i.d. observations $(Y_i^\ast,X_{ij})$ with
	\[
	E[Y^\ast_i | X_{ij} ] =  m_j(X_{ij}), \quad Y^\ast_i = Y_i-\sum_{k \neq j} m_k({X_{ik}}).
	\]
	In this respect the local linear estimator differs from other smoothing methods where the asymptotic bias of the estimator of the function $m_j$ depends on the shape of the functions $m_k$ for $k\not = j$. An example for an estimator with this disadvantageous bias property is the local constant smooth backfitting estimator which is based on a backfitting implementation of one-dimensional Nadaraya-Watson estimators. It is also the case for other smoothing estimators as regression splines, smoothing splines and orthogonal series estimators, where in addition also no closed form expression for the asymptotic bias is available.
	Asymptotic properties of local linear smoothing simplify the choice of bandwidths as well as the statistical interpretation of the estimators $\hat m_j$. These aspects have made local linear smooth backfitting a preferred choice for estimation in additive models.
Deriving asymptotic theory for local linear smooth backfitting is typically complicated by an overloaded notation that is required for detailed proofs. In this note we will use that the local linear smooth backfitting estimator has a nice geometric interpretation. This simplifies mathematical arguments and allows for a more intuitive derivation of asymptotic properties. In particular, we will see that the estimator can be characterized as a solution of an empirical integral equation of the second kind as is the case for local constant smooth backfitting, see \cite{mammen2009nonparametric}.
	
	Our main point is that the local linear estimator can be seen as an orthogonal projection of the response vector $Y= (Y)_{i=1,\dots,n}$ onto a subspace of a suitably chosen linear space. A similar point of view is taken in \cite{mammen2001general} for a related construction where it was also shown that regression splines, smoothing splines and orthogonal series estimators can be interpreted as projection of the data in an appropriately chosen Hilbert space. Whereas this interpretation is rather straight forward for these classes of estimators it is not immediately clear that it also applies for kernel smoothing and local polynomial smoothing, see \cite{mammen2001general}. In this paper we will introduce a new and simple view of local linear smoothing as data projection.  In the next section we will define the required spaces together with a corresponding semi-norm. We will also introduce a new algorithm motivated by our interpretation of local linear smooth backfitting. The algorithm will be discussed in Section \ref{sec:exist}. In Section \ref{sec:asymp} we will see that our geometric point of view allows for simplified arguments for the asymptotic study of properties of the
	local linear smooth backfitting estimator.
	
%

	The additive model \eqref{model:add} was first introduced in \cite{friedman1981projection} and enjoys great popularity for two main reasons.
	The first is estimation performance.
	While not being as restrictive as a linear model, in contrast to a fully flexible model, it is not subject to the curse of dimensionality. Assuming that  $E[Y_i | X_i =x]$ is twice continuously differentiable, the optimal rate of convergence of an estimator of $E[Y_i | X_i =x]$ is $n^{-2/(d+4)}$ if no further structural assumptions are made, see  \cite{stone1982optimal}. This means the rate deteriorates exponentially in the dimension of the covariates $d$. Under the additive model assumption \eqref{model:add} and assuming that each function $m_j$, $j=1,\dots, d$ is twice continuously differentiable, the optimal rate of convergence is $n^{-2/5}$.
	The second reason is interpretability. In many applications it is desirable to understand the relationship between predictors and the response. 
	Even if the goal is prediction only, understanding this relationship may help detect systematic biases in the estimator, so that out of sample performance can be improved or adjusted for.
	While it is almost impossible to grasp the global structure of a multivariate function $m$ in general, the additive structure  \eqref{model:add} allows for visualisation of each of the univariate functions, providing a  comprehensible connection between predictors and the response.

	Though the setting considered in this paper is fairly simple, it can be seen as a baseline for more complicated settings.One main drawback is the additive structure which cannot account for interactions between covariates.
 It is assuring however that even if the true model is not additive, the smooth backfitting estimator is still defined as the closest additive approximation.  This will be shown in the next section.
If the true regression function is far away from an additive structure,  then a more complex structure may be preferable. This could be done by adding higher-dimensional covariates, products of univariate functions or considering a generalized additive model. For testing procedures that compare such specifications, see also \cite{Hardle2001,MamSper21}. Besides such structural assumptions, other directions the ideas in this paper can be extended to are the consideration of time-series data or high dimensional settings. Settings using more complicated responses like survival times, densities or other functional data may also be approached. Some of these cases have been considered, e.g., in \cite{mammen2003generalised},  \cite{yu2008smooth}, \cite{mammen2009nonparametric}, \cite{mammen2014additive}, \cite{han2018smooth}, \cite{MamSper21}, \cite{han2019additive}, \cite{jeon2020additive}, \cite{hiabu2020smooth} and \cite{Gregory2020optimal}.
	We hope that a better understanding of local linear estimation in this simple setting will help advance theory and methodology for more complicated settings in the future. 
	
	\section{Local linear  smoothing in additive models} \label{sec:local linear}
	
	The local linear smooth backfitting estimator $\hat{m}=(\hat{m}_0,\hat{m}_1,\dots,\hat{m}_d,\hat{m}_1^{(1)},\dots,\hat{m}_d^{(1)})$ is defined as the minimizer of the criterion
	\begin{eqnarray*} 
		&&S(f_0,\dots,f_d, f_1^{(1)},\dots,f_d^{(1)}) \\
		&& \qquad  =  n^{-1}\sum_{i=1}^n \int_{\mathcal X} \left \{ Y_i - f_0 -\sum_{j=1}^d f_j(x_j) -\sum_{j=1}^d f_j^{(1)}(x_j)  (X_{ij}-x_j)\right \} ^2 \\
		&& \qquad \times K_h^{X_i}(X_i-x) \mathrm dx
	\end{eqnarray*}
	under the constraint
	\begin{eqnarray} \label{eq:constr}
	\sum_{i=1}^n \int_{\mathcal X}  f_j(x_j)  K_h^{X_i}(X_i-x) \mathrm dx= 0
	\end{eqnarray}
	for $j=1,\dots,d$. The minimization runs over all values $f_0 \in \mathbb R$ and all functions $f_j, f_j^{(1)}: \mathcal X_j \to \mathbb R$ with $\mathcal X_j = \{ u \in \mathbb R:$ there exists an $ x \in \mathcal X$ with $ x_j = u\}$. Under the constraint (\ref{eq:constr}) and some conditions introduced in Section \ref{sec:exist}, the minimizer is unique. For $j=1,\dots, d$ the local linear estimator of $m_j$ is defined by $\hat{m}_j$.	
	
	\noindent In the definition of $S$ the function $K_h^{u}(\cdot)$ is a boundary corrected product kernel, i.e.,
	\[
	K_h^{u}(u-x)=  \frac{\prod_{j=1}^d  \kappa \left(\frac{u_j-x_j}{h_j}\right)}{\int_{\mathcal X}  \prod_{j=1}^d  \kappa \left(\frac{u_j-v_j}{h_j}\right) \mathrm dv_j   }.
	\]
Here, $h=(h_1,\dots,h_d)$ is a bandwidth vector with $h_1,\dots,h_d > 0$ and $\kappa :\mathbb{R}\rightarrow\mathbb{R}$ is some given univariate density function, i.e., $\kappa (t)\geq 0$ and $\int \kappa (t) \mathrm dt=1$.
	We use the variable $u$ twice in the notation because away from the boundary of $\mathcal X$, the kernel  $K_h^{u}(u-x)$ only depends on $u-x$.\\
	It is worth emphasizing  that the empirical minimization criterion $S$ depends on a choice of a kernel $\kappa$ and a smoothing bandwidth $h$. While the choice of $\kappa$ is not of great importance, see similar to e.g. \cite[Section 3.3.2]{silverman2018density},
	the quality of estimation heavily depends on an appropriate choice of the smoothing parameter $h$.
	We will not discuss the choice of a (data-driven) bandwidth in this paper,  but we note that
	the asymptotic properties of the  local linear smoothing estimator do simplify the choice of bandwidths compared to other estimators. The reason is that the asymptotic bias of one additive  component does not depend on the shape of the other components and on the bandwidths used for the other components.

	\noindent We now argue that the local linear smooth backfitting estimator can be interpreted as an empirical projection of the data onto a space of additive functions. We introduce the linear space
	\[
	\mathcal H = \left \{ (f^{i,j})_ {i=1,\dots,n;\ j=0,\dots,d} | \  f^{i,j}:  \mathcal X \mapsto \mathbb R, \  {\norm f }_n< \infty \right \}
	\]
	with inner product
	\begin{eqnarray*} 
		{\innerproduct fg}_n &=& n^{-1}\sum_{i=1}^n \int_{\mathcal X} \left \{f^{i,0}(x) +\sum_{j=1}^df^{i,j}(x) (X_{ij}-x_j)\right \} \\
		&& \qquad \times \left \{g^{i,0}(x) +\sum_{k=1}^d g^{i,k}(x) (X_{ij}-x_j)\right\} K_h^{X_i}(X_i-x) \mathrm dx
	\end{eqnarray*} 
	and norm $\| f\|_n = \sqrt {{\innerproduct f f}_n}$.\newline
	We identify the response $Y= (Y_i)_{i=1,\dots,n}$ as an element of $\mathcal H$ via $Y^{i,0}\equiv Y_i$ and $Y^{i,j}\equiv 0$ for $j\geq1$.
	We will later assume that the functions $m_j$ are differentiable. We identify the regression function 
	\[
	m:\mathcal{X}\rightarrow\mathbb{R},\ m(x)=m_0+m_1(x_1)+\dots+m_d(x_d)
	\]
	as an element of 
	$\mathcal H$ via $m^{i,0}(x)=m_0+\sum_j m_j(x_j)$ and $m^{i,j}={\partial m_j(x_j)}/{\partial x_j}$ for $j\geq 1$.  Note that the components of $m\in\mathcal{H}$ do not depend on $i$.
	We define the following subspaces of $\mathcal H$:
	\begin{align*}
	\mathcal H_{full} &= \left \{ f \in \mathcal H | \text{ the components of }\ f \  \text{do not depend on} \ i \right \}, \\
	\mathcal H_{add} &=  \left \{ f \in \mathcal H_{full} | \  f^{i,0}(x)=f_0 + f_1(x_1)+\cdots + f_d(x_d),   f^{i,j}(x) =f^{(1)}_j(x_j)  \text{ for}\right.  \\ & \qquad   \text{some } f_0 \in \mathbb R \text { and some univariate functions}\  f_j, f_j^{(1)}: \mathcal X_j \to \mathbb R, j=1,\dots,\\
	&\qquad \left.  d\text{ with }\sum_{i=1}^n \int_{\mathcal X} f_{j}(x_j)  K^{X_i}_h(X_i-x) \mathrm dx =0  \right \}.
	\end{align*}
	For a function $f \in \mathcal H_{add}$ we write $f_0 \in \mathbb R$ and $f_j, f_j^{(1)}: \mathcal X_j \to \mathbb R$ for $j=1,\dots,d$ for the constant and functions that define $f$.  In the next section we will state conditions under which the constant $f_0$ and functions $f_j, f_j^{(1)}$ are unique given any $f \in \mathcal H_{add}$. 	By a slight abuse of notation we also write $f_j$ for the element of $\mathcal H$ given by $f^{i,0}(x)= f_j(x_j)$ and  $f^{i,k}(x) \equiv 0$ for $k=1,\dots,d$. We also write  $f_j^{(1)}$ for the element of $\mathcal H$ with $f^{i,k}(x) \equiv 0$ for $k \not = j$ and $f^{i,j}(x) = f_j^{(1)}(x_j)$. Furthermore, we define $f_{j+d}:=f_j^{(1)}$ for $j=1,\dots,d$ for both interpretations. Thus, for $f \in \mathcal H_{add}$ we have 
	\begin{align}\label{eq:add}
	f= f_0+\dots+f_{2d}.
	\end{align}
	Recall that the linear smooth backfitting estimator $$\widehat m=(\hat m_0, \hat m_1, \dots,\hat m_d, \hat m_1^{(1)},\dots,\hat m_d^{(1)})$$ is defined as the minimizer of the criterion $S$ under the constraint 
	\eqref{eq:constr}.
	By setting $\hat{m}^{i,0}(x)=\hat{m}_0+\sum_{j=1} ^d \hat{m}_j(x_j)$ and $m^{i,j}(x)=\hat{m}_j^{(1)}(x_j)$ for $j\geq 1$ it can easily be seen that
	\begin{align}\label{minimiser}
	\widehat m = \argmin_{f\in \mathcal H_{add}} {\norm{Y-f}}_n.
	\end{align}
	In the next section we will state conditions under which the minimization has a unique solution.
	Equation \eqref{minimiser} provides a geometric interpretation of local linear smooth backfitting. The local linear smooth backfitting estimator is an orthogonal projection of the response vector $Y$ onto the linear subspace $\mathcal H_{add}\subseteq\mathcal H$. We will make repeated use of this fact in this paper.\\
	We now introduce the following subspaces of $ \mathcal H$:
	\begin{align*}
	\mathcal H_0 &= \left\{ f \in \mathcal H | f^{i,0}(x)\equiv c \text{ for some } c \in \mathbb R, f^{i,j}(x)\equiv 0  \ \text{ for }\  j\neq0  \right \},\\
	\mathcal H_k &= \left\{ f \in \mathcal H  | \ f^{i,j}(x)\equiv 0 \ \text{ for }\ j\neq0, \text { and } f^{i,0}(x)=f_k(x_k)  \text{ for some  univariate}\right .\\
	& \left . \qquad  \text{function } f_k: \mathcal X_k \to \mathbb R \text{ with }
	\sum_{i=1}^n \int_{\mathcal X} f_{k}(x_k)  K^{X_i}_h(X_i-x) \mathrm dx =0 \right\},\\
	\mathcal H_{k'} &= \left\{ f \in \mathcal H  | \ f^{i,j}(x)\equiv 0 \ \text{ for }\ j \not =k,  f^{i,k}(x)=f^{(1)}_k(x_k)  \text{ for some  univariate}\right .\\
	& \left . \qquad  \text{function } f^{(1)}_k: \mathcal X_k \to \mathbb R  \right\}
	\end{align*}
	for $k=1,\dots,d$ and $k':=k+d$.
	Using these definitions we have $\mathcal H_{add}=\sum_{j=0} ^{2d} \mathcal H_j$ with $\mathcal H_j\cap \mathcal H_k=\{0\}, j\neq k$. In particular, the functions $f_j$ in \eqref{eq:add} are unique elements in $\mathcal H_j, j=0,\dots,2d.$
	For $k=0,\dots,2d$ we denote the orthogonal projection of $ \mathcal H$ onto the space $\mathcal H_k$ 
	by $\mathcal P_k$.
	Note that for $k=0,\dots,d$ the operators $\mathcal P_k$ set all components of an element $f = (f^{i,j})_ {i=1,\dots,n;\ j=0,\dots,d}\in\mathcal{H}$ to zero except the components with indices $(i,0), i=1,\dots,n$. Furthermore, for $k=d+1,\dots, 2d$, only components with index $(i,k-d)$  are not set to zero.
	Because $\mathcal H_0 $ is orthogonal to $\mathcal H_k$ for $k=1,\dots,d$, the orthogonal projection onto the space $\mathcal H_k$ is  given by $\mathcal P_k = P_k - \mathcal P_0$ where $P_k$ is the projection onto $\mathcal H_0+\mathcal  H_k$. In Appendix \ref{sec: projection operators} we will state explicit formulas for the orthogonal projection operators.\newline
	The operators $\mathcal P_k $ can be used to define an iterative algorithm for the approximation of $\hat m$. For an explanation observe that $\hat m$ is the projection of $Y$ onto $\mathcal H_{add}$ and $\mathcal H_k$ is a linear subspace of $\mathcal H_{add}$. Thus $\mathcal{P}_k(Y) = \mathcal{P}_k(\hat m) $ holds for $k=0,\dots, 2d$. This gives
	\begin{equation} \label{eq:back} \mathcal P_k(Y) =\mathcal  P_k(\hat m) = \mathcal  P_k\left ( \sum_{j=0} ^{2d} \hat m_j\right) =  \hat m_k + \mathcal  P_k\left( \sum_{j\not = k}  \hat m_j\right)\end{equation} 
	or, equivalently,
	$$\hat m_k = \mathcal P_k (Y) - \sum_{j\not = k} \mathcal  P_k(  \hat m_j)= P_k (Y) - \bar Y - \sum_{j\not = k} \mathcal  P_k(  \hat m_j),$$	
	where $\bar Y=\mathcal P_0 (Y)=P_0 (Y)$ is the element of   $\mathcal H$ with $(\bar Y)^{i,0}\equiv \frac 1 n \sum_{i=1}^n Y^i$,   $(\bar Y)^{i,j}\equiv 0$ for $j\geq1$. 
	This equation inspires an iterative algorithm where in each step approximations $\hat m^{old}_k$ of  $\hat m_k$ are updated by 
	$$\hat m_k^{new} =  P_k (Y) -\bar Y - \sum_{j\not = k} \mathcal  P_k(  \hat m^{old}_j).$$

	\begin{algorithm}
		\caption{Smooth Backfitting algorithm} 
		\begin{algorithmic}[1]
			\State \textbf{Start:} $\hat m_k(x_k)\equiv 0, \widetilde m_k =\mathcal P_k( Y), error=\infty$ \Comment{$k=0,\dots,2d$}
			\While {$error>tolerance$}
			\State $error \gets 0$
			\For{$k=0,\dots,2d$}
			\State $\hat m_k^{old} \gets \hat m_k$
			\State $\hat m_k \gets  \widetilde m_k  - \sum_{j\not = k} \mathcal  P_k(  \hat m_j)$ 			
			\State $error \gets error + | \hat m_k - \hat m_k^{old}|$
			\EndFor\label{euclidendwhile}
			\EndWhile
			\State \textbf{return} $\hat m=(\hat m_0, \hat m_1,\dots, \hat m_{2d})$
		\end{algorithmic}
	\end{algorithm}

	Algorithm 1 provides a compact definition of our algorithm for the approximation of $\hat m$. In each iteration step, either $\hat m_j$ or $ \hat m_j^{(1)}$ is updated for some $j=1,\dots, d$. This is different from the algorithm proposed in \cite{mammen1999existence} where in each step  a  function tuple  $(\hat m_j, \hat m_j^{(1)})$ is updated. For the orthogonal projections of functions $m \in \mathcal H_{add} $ one can use simplified formulas. They will be given  in Appendix \ref{sec: projection operators}.
	Note that $\tilde m_k, k=0,\dots,2d$ only needs to be calculated once at the beginning. Also the marginals 
	$p_k(x_k)$, $p^\ast_k(x_k)$, $p^{\ast \ast}_k(x_k)$, $p_{jk}(x_j,x_k)$, $p^{\ast}_{jk}(x_j, x_k)$ and $p^{\ast \ast}_{jk}(x_j,x_k)$ which are needed in the evaluation of $\mathcal P_k$ only need to be calculated once at the beginning. Precise definitions of these marginals can be found in the following sections. In each iteration of the for-loop in line 4 of Algorithm 1, $O(d \times n \times  gs)$ calculations are performed. Hence for a full cycle, the algorithm needs $O(d^2 \times n \times gs\times \log (1/ \mathrm{tolerance}))$ calculations. Here $\mathrm{grid.size}$ is the number of evaluation points for each coordinate $x_k$. \newline

	Existence and uniqueness of the local linear smooth backfitting estimator will be discussed in the next section. Additionally, convergence of the proposed iterative algorithm will be shown.

	\section{Existence and  uniqueness of the estimator,  convergence of the algorithm} \label{sec:exist}
 In this section we will establish conditions for existence and uniqueness of the local linear smooth backfitting estimator $\widehat m$. Afterwards we will discuss convergence of the iterative algorithm provided in Algorithm 1. Note that convergence is shown for arbitrary starting values, i.e., we can set $\hat{m}_k(x_k)$ to values other than zero in step 1 of Algorithm 1. For these statements we require the following weak condition on the kernel.
	\begin{enumerate}
		\item[(A1)]    
		The kernel $k$ has support $[-1,1]$. Furthermore, $k$ is strictly positive on $(-1,1)$ and continuous on $\mathbb R$.
	\end{enumerate}

	For $k=1,\dots,d$ and $x\in \mathbb{R}^d$ we write
	\[
	x_{-k}:=(x_1,\dots,x_{k-1},x_{k+1},\dots,x_d).
	\]
	In the following, we  will show that our claims hold on the following event:
	\begin{align*}
	\mathcal E &= \bigg \{ \text{ For } k=1,\dots,d \text{ and } x_k\in \overline{\mathcal {X}}_k  \text{ there exist two observations}\  i_1,i_2 \in \{1,\dots,n\} \\ 
	& \text{such that} 
	\   X_{i_1,k}  \not = X_{i_2,k}\ ,\ |X_{i,k}  - x_k| < h_k, (x_k,X_{i,-k})\in\overline{\mathcal{X}}\ \text{for}\ i=i_1,i_2.\\ 
	&\text{Furthermore, there exist no } b_0,\dots,b_d \in \mathbb R  \text{ with } b_0 +
	\sum_{j=1} ^d b_j X_{ij} = 0\ \forall i=1,\dots,n
	\bigg  \},
	\end{align*}
	where $\overline{\mathcal{X}}_k$ is the closure of $\mathcal{X}_k$ and by a slight abuse of notation $$(x_k,X_{i,-k}):=(X_{i,1},\dots, ,X_{i,k-1},x_k,X_{i,k+1},\dots,,X_{i,d}).$$
	Throughout this paper, we require the following definitions.
	\begin{eqnarray*} 
		\hat p_k(x_k) &=& \frac  1 n \sum_{i=1} ^n \int _{\mathcal X_{-k}(x_k) } K^{X_i}_h(X_i-x)\mathrm dx_{-k}, \\
		\hat p^{*}_k(x_k) &=& \frac  1 n \sum_{i=1} ^n \int _{\mathcal X_{-k}(x_k) } (X_{ik}-x_k)  K^{X_i}_h(X_i-x)\mathrm dx_{-k},\\
		\hat p^{**}_k(x_k) &=& \frac  1 n \sum_{i=1} ^n \int _{\mathcal X_{-k}(x_k) } (X_{ik}-x_k) ^2 K^{X_i}_h(X_i-x)\mathrm dx_{-k},
	\end{eqnarray*}
	where ${\mathcal X_{-k}(x_k) }=\{ u_{-k}\ |\ (x_k,u_{-k}) \in \mathcal X \}$.

	\begin{lemma} \label{lem:uni Hadd}
		Make Assumption (A1). Then, on the event $\mathcal E$ it holds that $\|f\|_n =0$ implies $f_0 =0$ as well as $f_j\equiv 0$ almost everywhere for $j=1\dots 2d$ and all $f \in \mathcal H_{add}$. 	\end{lemma}
	
	One can easily see that the lemma implies the following. On the event $\mathcal{E}$, if a minimizer $\hat m=\hat m_0+\cdots +\hat m_{2d}$ of ${\norm{Y-f}}_n$ over $f \in \mathcal H_{add}$ exists, the components $\hat m_0,\dots, \hat m_{2d}$ are uniquely determined: Suppose there exists another minimizer $\tilde m \in \mathcal H_{add}$. Then it holds that $\langle Y- \hat m, \hat m - \tilde m\rangle_n = 0$ and $\langle Y- \tilde m, \hat m - \tilde m\rangle_n = 0$ which gives $\|\hat m - \tilde m\|_n =0$. An application of the lemma yields uniqueness of the components  $\hat m_0,\dots, \hat m_{2d}$.
	\begin{figure}
		\centering
		\includegraphics[width=5cm]{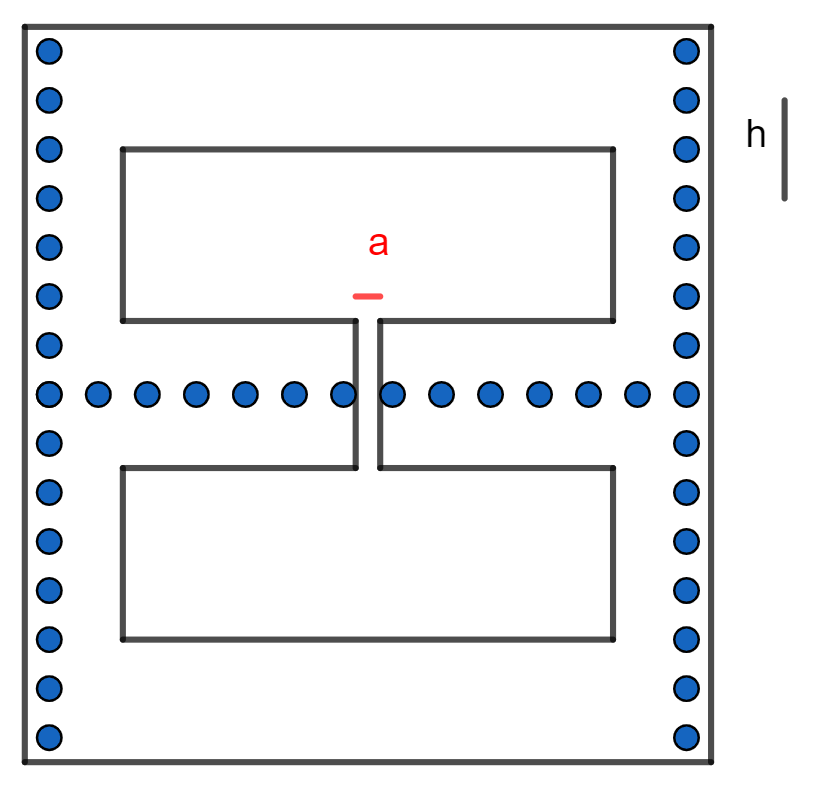}
		\caption{An example of a possible data set $\mathcal{X}\subseteq\mathbb{R}^2$ including data points where the conditions of the event $\mathcal E$ are not satisfied and where the components of functions $f\in \mathcal H_{add}$ are not identified. The data is visualized by blue dots. The size of the parameter $h=h_1=h_2$ is showcased on the right hand side. For explanatory reasons, the interval $a$ is included.} \label{fig:addsmooth}
	\end{figure}
	\begin{remark}
		In Figure 1 we give an example where a set $\mathcal X$ and data points $X_i$  do not belong to the event $\mathcal E$ and where the components of the function $f\in \mathcal H_{add}$ are not identified.
		Note that in this example for all $k=1,2$ and $x_k\in\overline{\mathcal{X}_k}$ there exist $i_1,i_2\in\{1,\dots,n\}$ such that$X_{i_1}\neq X_{i_2}$ and  $|X_{i,k}-x_k|<h$, for $i=i_1,i_2$. However, for $x_1\in a$ the condition $(x_1,X_{i,2})\in\overline{\mathcal{X}}$ is not fulfilled for any $i=1,\dots,n$ with $|X_{i,1}-x_1|<h$. Therefore, $K_h^{X_i}(X_i-x)=0$ for all $x\in\overline{\mathcal{X}}$ with $x_1\in a$. Thus, any function satisfying $f\in\mathcal{H}_1$ with $f_1(x)=0$ for $x\in\mathcal{X}_1\backslash\{a\}$ has the property $\norm{f}_n=0$.  
	\end{remark}
	
	\begin{proof} [of Lemma \ref{lem:uni Hadd}]
		
		First, for each pair $i_1,i_2=1,\dots,n$ define the set
		\[
		M_{i_1,i_2}:=\{x_1\in\mathcal{X}_1\ |\ |X_{i,1}  - x_1| < h, (x_1,X_{i,-1})\in\mathcal{X}\ \text{for}\ i=i_1,i_2\}
		\]
		if $X_{i_1,1}\neq X_{i_2,1}$ and $M_{i_1,i_2}=\emptyset$ otherwise.
		It is easy to see that $M_{i_1,i_2}$ is open as an intersection of open sets. Note that on the event $\mathcal E$ we have
		\begin{equation} \label{eq:Ueberdeckung}
		\bigcup_{i_1,i_2}M_{i_1,i_2}=\mathcal{X}_1.
		\end{equation}
		Now, suppose that for some $f \in \mathcal H_{add}$ we have $\|f\|_n =0$. We want to show that $f_0 =0$ and that $f_j \equiv 0$ for $j=1,\dots,2d$. From $\|f\|_n =0$ we obtain
		$$ \bigg \{ f_0 + \sum_{j=1} ^d f_j(x_j) + \sum_{j=1} ^d f_{j'}(x_j) (X_{ij}-x_j)\bigg \}^2 K^{X_i}_h(X_i -x) = 0$$
		for $i=1,\dots, n$ and almost all $x \in \mathcal X$. Let $i_1,i_2\in\{1,\dots,n\}$. Then
		\begin{equation} \label{eq:Haddhelp} f_0 + f_1(x_1) + \sum_{j=2} ^d f_j(X_{ij}) + f_{d+1}(x_{1}) (X_{i1}-x_1) =0\end{equation} 
		holds for all $i=i_1,i_2$ and $x_1\in M_{i_1,i_2}$ almost surely. By subtraction of Equation \eqref{eq:Haddhelp} for $i=i_1$ and $i=i_2$ we receive
		$$f_{d+1}(x_{1}) = v_1$$
		with constant $v_1 = -  \sum_{j=2} ^d (f_j(X_{i_1,j}) - f_j(X_{i_2,j})) / (X_{i_1,1} - X_{i_2,1} )$ for $x_1\in M_{i_1,i_2}$. 
		Furthermore, by using \eqref{eq:Haddhelp} once again we obtain
		$$f_{1}(x_{1}) = u_1 + v_1 x_1$$
		with another constant $u_1 \in \mathbb R$. Following \eqref{eq:Ueberdeckung}, since $\mathcal{X}_1$ is connected and the sets $M_{i_1,i_2}$ are open we can conclude
		$$f_{d+1}(x_{1}) = v_1 \text{ and } f_{1}(x_{1}) = u_1 + v_1 x_1$$
		for almost all $x_1 \in \mathcal X_1$ since the sets must overlap. Similarly one shows $$f_{j'}(x_{j}) = v_j \text{ and } f_{j}(x_{j}) = u_j + v_j x_j$$
		for $j =2,\dots, d$ and almost all $x_j \in \mathcal X_j$.  We conclude that  
		\begin{eqnarray*} 0 &=& \|f\|_n ^2 \\
			&=& \frac 1 n \int \sum_{i=1} ^n \bigg \{ f_0 + \sum_{j=1} ^d f_j(x_j) + \sum_{j=1} ^d f_{j'}(x_j) (X_{ij}-x_j)\bigg \}^2 K^{X_i}_h(X_i -x) \mathrm d x\\
			&=& \frac 1 n  \int \sum_{i=1} ^n \bigg \{ f_0 + \sum_{j=1} ^d u_j + \sum_{j=1} ^d v_jX_{ij}\bigg \}^2 K^{X_i}_h(X_i -x) \mathrm d x\\
			&=&  \bigg \{ f_0 + \sum_{j=1} ^d u_j + \sum_{j=1} ^d v_jX_{ij}\bigg \}^2 .
		\end{eqnarray*}
		On the event $\mathcal E $ the covariates $X_i $ do not lie in a linear 
		subspace of $ \mathbb R^d$. This shows $v_j=0$ for $1\leq j\leq d$. Thus $f_j \equiv 0$ for $d+1 \leq j \leq 2d$ and $f_j = u_j$ for $1 \leq j \leq d$.\newline
		
		\noindent Now, $\int f_j(x_j) \hat p_j(x_j) \mathrm d x_j = 0$ implies that $f_j \equiv 0$ for $1 \leq j \leq d$ and $f_0=0$. This concludes the proof of the lemma.
	\end{proof}

	Existence and uniqueness of $\widehat m$ on the event $\mathcal E$ under Assumption (A1) follows immediately from the following lemma.
	
	\begin{lemma} \label{lem:Hadd closed}
		Make Assumption (A1). Then, on the event $\mathcal E$, for every $D\subseteq\{0,\dots,2d\}$ the linear space 
		$\sum_{k\in D}\mathcal{H}_k$ is a closed subset of $\mathcal H$. In particular, $\mathcal H_{add}$ is closed.
	\end{lemma}

	For the proof of this lemma we make use of some propositions introduced below. In the following, we consider sums $L=L_1 + L_2$ of closed subspaces $L_1$ and $ L_2$ of a Hilbert space with $L_1\cap L_2=\{0\}$. In this setup, an element $g\in L$ has a unique decomposition $g=g_1+g_2$ with $g_1\in L_1$ and $g_2\in L_2$. Thus, the projection operator from $L$ onto $L_1$ along $L_2$ given by 
	\begin{align*}
	\Pi_1(L_2):L\rightarrow L_1,\ \Pi_1(L_2)(g)=g_1
	\end{align*}
	is well defined. 
	
	\begin{proposition} 
		For the sum $L=L_1 + L_2$ of two closed subspaces $L_1$ and $L_2$ of a Hilbert space with $L_1\cap L_2=\{0\}$,
		the following conditions are equivalent
		\begin{enumerate}
			\item[(i)] $L$ is closed.
			\item[(ii)] There exists a constant $c>0$ such that for every $g=g_1+g_2\in L$ with $g_1 \in L_1$ and $g_2 \in L_2$ we have
			\begin{equation}\label{cond2:closed}
			\|g\| \geq c \max \{\|g_1\|,  \|g_2\|\}.
			\end{equation}
			\item[(iii)] The projection operator $\Pi_1(L_2)$ from $L$ onto $L_1$ along $L_2$ is bounded. 
			\item[(iv)] The gap from  $L_1$ to $L_2$ is greater than zero, i.e.,
			\[\gamma(L_1,L_2):=
			\inf_{g_1 \in L_1} \frac{dist(g_1,L_2)}{\norm{g_1}} >0,
			\]
			where $\mathrm{dist}(f,V):= \inf_{h \in V} \norm{f-h}$ with the convention $0/0=1$.
		\end{enumerate}
		\label{lprop:Hadd}
	\end{proposition} 
	\begin{remark}
		A version of Proposition \ref{lprop:Hadd} is also true if $L_1\cap L_2\neq\{0\}$. In this case, the quantities involved need to be identified as objects in the quotient space $L/(L_1\cap L_2)$.
	\end{remark}
	
	\begin{proposition} \label{lprop:Hadd closed help1}
		The sum $L=L_1 + L_2$ of two closed subspaces $L_1$ and $L_2$ of a Hilbert space with $L_1\cap L_2=\{0\}$ is closed if the orthogonal projection of $L_2$ on $L_1$ is compact.\end{proposition} 
	
	The  proofs of Propositions  \ref{lprop:Hadd} and \ref{lprop:Hadd closed help1} can be reconstructed from \cite[A.4 Proposition 2]{bickel1993efficient}, \cite[Chapter 4, Theorem 4.2]{kato2013perturbation} and \cite{kober1940theorem}. For completeness,  we have added proofs of the propositions in Appendix \ref{app:B}.\newline
	
	\noindent We now come to the proof of Lemma \ref{lem:Hadd closed}.
	\begin{proof} [of Lemma \ref{lem:Hadd closed}] First note that the spaces $\mathcal H_k$ are closed for $k=0,\dots, 2d$.\newline

		\noindent We show that $\mathcal H_k + \mathcal H_{k'}$ is closed for $1 \leq k \leq d$.  Consider
		$R= \min M$ where 
		\[
		M:=\{ r\geq 0\ |\ (\hat p_k^*(x_k))^2 \leq r \hat p_k(x_k) \hat p_k^{**}(x_k) \text{ for all } x_k\in\overline{\mathcal{X}}_k \text{ and } 1 \leq k \leq d\} 
		\]
		and 
		\[
		I_{x_k}:=\bigg\{i\in\{1,\dots,n\}\ \bigg|\ \int_{u \in \mathcal X_{-k}(x_k) }K^{X_i}_h(X_i-u)\mathrm du_{-k} > 0\bigg\}
		\]
		for $x_k\in\overline{\mathcal X}_k$. By the Cauchy-Schwarz inequality we have $(\hat p_k^*(x_k))^2 \leq \hat p_k(x_k) \hat p_k^{**}(x_k)$ for $x_k\in\overline{\mathcal{X}}_k$ and $1\leq k\leq d$. This implies $R \leq 1$ . Now, equality in the inequality only holds if $X_{ik} - x_k$ does not depend on $i\in I_{x_k}$. On  the event $\mathcal E$ for $x_k\in \overline{\mathcal{X}}_k$ there exist $1\leq i_1,i_2\leq n$ with $|x_k-X_{i,k}|<h$ for $i=i_1,i_2$ and $X_{i_1,k}\neq X_{i_2,k}$. Thus, $X_{ik}-x_k$ depends on $i$ for $i\in I_{x_k}$ and the strict inequality holds for all $x_k$. Furthermore, 
		because the kernel function $k$ is continuous, we have that $\hat p_k$, $\hat p_k^{*}$ and $\hat p_k^{**}$ are continuous. Thogether with the compactness of $\overline{\mathcal{X}}_k$ this implies that $R < 1$ on  the event $\mathcal E$.\newline
		Now let $f \in \mathcal H_k$ and $g \in  \mathcal H_{k'}$ for some $1\leq k\leq d$. We will show
		\begin{equation} \label{eq: 2 closed} 
		\|f+g\|_n^2 \geq (1-R) (\|f\|_n^2 + \|g\|_n^2).
		\end{equation}
		By application of Proposition \ref{lprop:Hadd} this immediately implies that $\mathcal H_k + \mathcal H_{k'}$ is closed. For a proof of \eqref{eq: 2 closed} note that
		\begin{eqnarray*}
			\|f+g\|_n^2 &=& n^{-1} \sum_{i=1} ^n (f_k(x_k) + g_0 + (X_{ik} -x_k) g_{k'}(x_k))^2 K_h^{X_i}(X_i -x) \mathrm d x\\
			&=&  \int ( f_k(x_k) +g_0)^2 \hat p_k(x_k) \mathrm d x_k  + 
			2 \int ( f_k(x_k) + g_0) g_{k'}(x_k) \hat p^{*}_k(x_k)  \mathrm d x_ k \\
			&& \qquad+  \int g_{k'}(x_k)^2 \hat p^{**}_k(x_k)  \mathrm d x_ k\\
			&\geq& \int ( f_k(x_k) +g_0)^2 \hat p_k(x_k) \mathrm d x_k  +  \int g_{k'}(x_k)^2 \hat p^{**}_k(x_k)  \mathrm d x_ k
			\\
			&& \qquad - 
			2 R \int | f_k(x_k) + g_0| |g_{k'}(x_k)| (\hat p_k(x_k)\hat p^{**}_k(x_k))^{1/2}  \mathrm d x_ k\\
			&\geq& \int ( f_k(x_k) +g_0)^2 \hat p_k(x_k) \mathrm d x_k  +  \int g_{k'}(x_k)^2 \hat p^{**}_k(x_k)  \mathrm d x_ k
			\\
			&& -
			2 R  \left (  \int ( f_k(x_k) +g_0)^2 \hat p_k(x_k) \mathrm d x_k  \right )^{1/2}  \left (  \int g_{k'}(x_k)^2 \hat p^{**}_k(x_k)  \mathrm d x_ k\right) ^{1/2} 
			\\
			&\geq& (1-R)  \int ( f_k(x_k) +g_0)^2 \hat p_k(x_k) \mathrm d x_k  + (1-R) \int g_{k'}(x_k)^2 \hat p^{**}_k(x_k)  \mathrm d x_ k
			\\
			&=& (1-R) \left ( \int f_k(x_k)^2 \hat p_k(x_k) \mathrm d x_k  +  g_0^2+ \int g_{k'}(x_k)^2 \hat p^{**}_k(x_k)  \mathrm d x_ k \right )\\
			&\geq&  (1-R) (\|f\|_n^2 + \|g\|_n^2),
		\end{eqnarray*}
		where in the second to last row, we used that $f\in\mathcal{H}_k$.
		This concludes the proof of \eqref{eq: 2 closed}.\newline
		Note that the statement of the lemma is equivalent to the following statement: For $D_1, D_2\subseteq\{1,\dots,d\}$ and $\delta\in\{0,1\}$ the space $\mathcal \delta\mathcal{H}_0+\sum_{k\in D_1} \mathcal{H}_k  +\sum_{k\in D_2}\mathcal{H}_{k'}$ is closed. We show this inductively over the number of elements $s=|D_2\cap D_1|$ of $D_1\cap D_2$.\newline 
		
		\noindent For the case $s=0$, note that
		for $D_1\cap D_2=\emptyset$, the space $\mathcal \delta\mathcal{H}_0+\sum_{k\in D_1} \mathcal{H}_k+\sum_{k\in D_2}\mathcal{H}_{k'}$ is closed, which can be shown with similar but simpler arguments than the ones used below. 
		Now let $s\geq 1$, $\delta\in\{0,1\}$, $D_1,D_2\subseteq \{1,\dots,d\}$ with $|D_2\cap D_1|=s-1$ and assume $L_2 = \delta\mathcal{H}_0+\sum_{j\in D_1}\mathcal{H}_j+\sum_{j\in D_2}\mathcal{H}_{j'}$ is closed. Without loss of generality, let $k\in\{1,\dots,d\}\backslash (D_1\cup D_2)$. We will argue that on the event $\mathcal E$ the orthogonal projection of $L_2$ on $L_1 = \mathcal H_{k}+ \mathcal H_{k'}$ is Hilbert-Schmidt, noting that a Hilbert-Schmidt operator is compact. 
		Using Proposition \ref{lprop:Hadd closed help1} since $L_1$ and $L_2$ are closed, this implies that $L = \mathcal \delta \mathcal{H}_0+\sum_{j\in D_1\cup\{k\}}\mathcal{H}_j+\sum_{j\in D_2\cup\{k\}}\mathcal{H}_{j'}$ is closed which completes the inductive argument.\newline
		For an element $f\in L_2$ with decomposition $f= f_0 + \sum_{j\in D_1}f_j+\sum_{j\in D_2}f_{j'}$ the projection onto $L_1+\mathcal{H}_0$ is given by univariate functions $g_k, g_{k'}$ and $g_0 \in \mathbb R$ which satisfy
		\begin{eqnarray*}
			0&=& \sum_{i=1} ^n \int \bigg (f_0 + \sum_{j\in D_1} f_j(x_j) + \sum_{j\in D_2} f_{j'}(x_j) (X_{ij}-x_j) {- g_0} - g_k(x_k) \\
			&& \qquad - g_{k'}(x_k) (X_{ik}-x_k) \bigg )\left(\begin{array}{c}1 \\X_{ik}-x_k\end{array}\right)
			K^{X_i}(X_i -x) \mathrm d x_{-k}.
		\end{eqnarray*}
		Note that $f_0=0$ if $\delta=0$. This implies
		\begin{eqnarray*}
			&&\left(\begin{array}{c}{g_0+}g_k(x_k)  \\g_{k'}(x_k)\end{array}\right)
			= \frac 1 {(\hat p_k \hat p_k^{**} - ( \hat p_k^{*})^2)(x_k)} \left(\begin{array}{cc}\hat p_k^{**} & -\hat p_k^{*} \\ -\hat p_k^{*} & \hat p_k\end{array}\right)(x_k)\\
			&&\qquad \times \bigg \{f_0 \left(\begin{array}{c}\hat p_k\\\hat p_k^{*}\end{array}\right) (x_k) +
			\sum_{j\in D_1}\int  f_j(x_j) \left(\begin{array}{c}\hat p_{jk} (x_j,x_k) \\ \hat p_{kj} ^{*}(x_k,x_j)\end{array}\right) \mathrm d x_{j}
			\\ && \qquad \qquad+\sum_{j\in D_2}\int  f_{j'}(x_j) \left(\begin{array}{c}\hat p^{*}_{jk} (x_j,x_k) \\ \hat p_{jk} ^{**}(x_j,x_k)\end{array}\right) \mathrm d x_{j}\bigg \},
		\end{eqnarray*}
		where $g_k$  and $g_0$ are  chosen such that $\int g_k(x_k) \hat p_k(x_k) \mathrm d x_k =0$. 
		We now use that the projection of $g_0$ onto $L_1 = \mathcal H_{k}+ \mathcal H_{k'}$ is equal to 
		\begin{eqnarray*}
			&&\left(\begin{array}{c}r_k(x_k)  \\r_{k'}(x_k)\end{array}\right)
			= \left(\begin{array}{c} g_0 (1 -c_k s_k(x_k) \hat p_k^{**}(x_k)) \\g_0 c_k s_k(x_k) $ $\hat p_k^{*} (x_k)\end{array}\right),
		\end{eqnarray*}
		where $s_k(x_k) =  \hat p_k(x_k)/ ( \hat p_k(x_k) \hat p_k^{**}(x_k)-( \hat p_k^{*})^2(x_k))$ and $c_k = (\int s_k(x_k) \hat p^{**}_k(x_k)$ $ \hat p_k(x_k) \mathrm d x_k) ^{-1} $.
		Thus,  the projection of $f$ onto $L_1$ is defined by $(g_{k}(x_k) + r_{k}(x_k),$ $ g_{k'}(x_k)+ r_{k'}(x_k)) ^\intercal$.
		Under our settings on the event $\mathcal E$ this is a Hilbert-Schmidt operator. This concludes the proof.
	\end{proof}

	We now come to a short discussion of the convergence of Algorithm 1. The algorithm is used to approximate $\hat m$. In the lemma we denote by 
	$\hat m^{[r]}$ the outcome of the algorithm after $r$ iterations of the while loop (see Algorithm 1).\newline

	\noindent We prove the algorithm for arbitrary starting values, i.e. we can set the $\hat{m}_k(x_k)$ to values other than zero in step 1 of Algorithm 1. The vector of starting values of the algorithm is denoted by $\hat m^{[0]}\in\mathcal{H}_{add}$.
	
	\begin{lemma} \label{lem:conv alg}
		Make Assumption (A1). Then, on the event $\mathcal E$, for Algorithm 1 and all choices of starting values $\hat m^{[0]}\in\mathcal{H}_{add}$ we have
		$$\| \hat m^{[r]} - \hat m \|_n \leq V ^r \| \hat m^{[0]} - \hat m \|_n,$$
		where $0 \leq V=1-\prod_{k=0}^{2d-1}\gamma^2(\mathcal{H}_k,\mathcal{H}_{k+1}+\dots+\mathcal{H}_{2d})< 1$  is a random variable depending on the observations. 	\end{lemma} 
	\begin{remark}
	On the event $\mathcal E$, the algorithm converges with a geometric rate where in every iteration step the distance to the limiting value, $\hat m$, is reduced by a factor  smaller or equal to $V$. If the columns of the design matrix $X$ are orthogonal, $V$ will be close to zero and if they are highly correlated, $V$ will be close to 1.
	The variable $V$ depends on $n$ and is random. Under additional assumptions, as stated in the next section, one can show that with probability tending to one, $V$ is bounded by a constant smaller than 1.
	\end{remark}
	\begin{proof} [of Lemma \ref{lem:conv alg}]
		For a subspace $\mathcal{V}\subseteq\mathcal{H}_{add}$ we denote by $\mathcal{P}_\mathcal{V}$ the orthogonal projection onto $\mathcal{V}$.
		For $k=0,\dots,2d$ let $\mathcal{Q}_k:=\mathcal{P}_{\mathcal{H}_k^\perp}=1-\mathcal{P}_k$ be the projection onto the orthogonal complement $\mathcal{H}_k^\perp$ of $\mathcal{H}_k$. The idea is to show the following statements.
		\begin{enumerate}
			\item[(i)] $Y-\hat{m}^{[r]}=(\mathcal{Q}_{2d}\dots \mathcal{Q}_0)^r(Y-\hat{m}^{[0]})$,
			\item[(ii)] $(\mathcal{Q}_{2d}\dots\mathcal{Q}_0 )^r(Y-\hat{m})=Y-\hat{m}$,
		\end{enumerate}
		This then implies 
		\begin{align*}
		\norm{\hat{m}^{[r]}-\hat{m}}_n &=\norm{\hat{m}^{[r]}-Y+Y-\hat{m}}_n= \norm{(\mathcal{Q}_{2d}\dots\mathcal{Q}_0)^r(\hat{m}^{[0]}-\hat{m})}_n.
		\end{align*}
		The proof is concluded by showing
		\begin{equation}\label{eq: Sinus}
		\norm{\mathcal{Q}_{2d}\dots\mathcal{Q}_0g}^2_n\leq \bigg(1-\prod_{k=0}^{2d-1}\gamma^2(\mathcal{H}_{k},\mathcal{H}_{k+1}+\dots+\mathcal{H}_{2d})\bigg)\norm{g}^2_n
		\end{equation}
		for all $g\in \mathcal{H}_{add}$. Note that $0\leq V:=1-\prod_{k=0}^{2d-1}\gamma^2(\mathcal{H}_k,\mathcal{H}_{k+1}+\dots+\mathcal{H}_{2d})<1$ by Lemma \ref{lem:Hadd closed} and Proposition \ref{lprop:Hadd}.\newline
		
		\noindent For (i), observe that for all $r\geq 1$ and $k=0,\dots,2d$ we have
		\begin{align*}
		& Y-\hat{m}^{[r-1]}_0-\dots-\hat{m}_{k-1}^{[r-1]}-\hat{m}_{k}^{[r]}-\dots-\hat{m}^{[r]}_{2d}\\
		& =(1-\mathcal{P}_k)(Y-\hat{m}^{[r-1]}_0-\dots-\hat{m}_{k-1}^{[r-1]}-\hat{m}_{k+1}^{[r]}-\dots-\hat{m}^{[r]}_{2d})\\
		& =(1-\mathcal{P}_k)(Y-\hat{m}^{[r-1]}_0-\dots-\hat{m}_{k-1}^{[r-1]}-\hat{m}_{k}^{[r-1]}-\hat{m}_{k+1}^{[r]}-\dots-\hat{m}^{[r]}_{2d})\\
		&=\mathcal{Q}_k(Y-\hat{m}^{[r-1]}_0-\dots-\hat{m}_{k}^{[r-1]}-\hat{m}_{k+1}^{[r]}-\dots-\hat{m}^{[r]}_{2d}).
		\end{align*}
		The statement follows inductively by beginning with the case $r=1, k=0$. Secondly, (ii) follows from
		\begin{align*}
		\mathcal{Q}_r\dots\mathcal{Q}_0(Y-\hat{m})& =\mathcal{Q}_r\dots\mathcal{Q}_0\mathcal{P}_{\mathcal{H}_0^\perp\cap\dots\cap \mathcal{H}_{2d}^\perp}(Y) = \mathcal{P}_{\mathcal{H}_0^\perp\cap\dots\cap \mathcal{H}_{2d}^\perp}(Y)=Y-\hat{m}.
		\end{align*}
		It remains to show the inequality in \eqref{eq: Sinus}.\newline
		
		\noindent For $0\leq k\leq 2d$ define $\mathcal{N}_k:=\mathcal{H}_{k}+\dots+ \mathcal{H}_{2d}$. We prove $\norm{\mathcal{Q}_{2d}\dots\mathcal{Q}_jg}_n^2\leq (1-\prod_{k=j}^{2d-1}\gamma^2(\mathcal{H}_{k},\mathcal{H}_{k+1}+\dots+\mathcal{H}_{2d}))\norm{g}_n^2$ for all $g\in\mathcal{H}_{add}$ and $0\leq j\leq 2d$ using an inductive argument.\newline 
		The case $j=2d$ is trivial. For $0\leq j<2d$ and any $g\in \mathcal{H}_{add}$ let $g_{j}^\perp:=\mathcal{Q}_jg=g'+g''$ with $g':=\mathcal{P}_{\mathcal{N}_{j+1}^\perp}(g)$ and $g'':=\mathcal{P}_{\mathcal{N}_{j+1}}(g)$. Then, by orthogonality, we have
		\[
		\norm{\mathcal{Q}_{2d}\dots\mathcal{Q}_{j+1}g_j^\perp}_n^2=\norm{g'+\mathcal{Q}_{2d}\dots\mathcal{Q}_{j+1}g''}^2_n=\norm{g'}^2_n+\norm{\mathcal{Q}_{2d}\dots\mathcal{Q}_{j+1}g''}^2_n.
		\]
		Induction gives
		\[
		\norm{\mathcal{Q}_{2d}\dots\mathcal{Q}_{j+1}g''}^2_n\leq \bigg(1-\prod_{k=j+1}^{2d-1}\gamma^2(\mathcal{H}_k,\mathcal{H}_{k+1}+\dots+\mathcal{H}_{2d})\bigg)(\norm{g^\perp_j}_n^2-\norm{g'}_n^2)
		\]
		which implies
		\begin{align*}
		\norm{\mathcal{Q}_{2d}\dots\mathcal{Q}_{j+1}g_j^\perp}_n^2\leq\bigg(& 1-\prod_{k=j+1}^{2d-1}\gamma^2(\mathcal{H}_k,\mathcal{H}_{k+1}+\dots+\mathcal{H}_{2d})\bigg)\norm{g^\perp_j}_n^2\\
		& +\prod_{k=j+1}^{2d-1}\gamma^2(\mathcal{H}_k,\mathcal{H}_{k+1}+\dots+\mathcal{H}_{2d})\norm{g'}_n^2.
		\end{align*}
		By Lemma \ref{lem:Hadd closed} and Lemma \ref{lemma gamma} we have 
		\[
		\norm{g'}_n^2\leq \norm{\mathcal{P}_{\mathcal{N}_{j+1}^\perp}\mathcal{Q}_1}_n^2=\norm{\mathcal{P}_{N_{j+1}}\mathcal{P}_{\mathcal{H}_j}}_n^2=1-\gamma^2(\mathcal{H}_j,\mathcal{H}_{j+1}+\dots+\mathcal{H}_{2d}).
		\]
		This concludes the proof by noting that $\norm{g_j^\perp}_n\leq\norm{g}_n$.
	\end{proof}
	
	\section{Asymptotic properties of the estimator} \label{sec:asymp}
	
	In this section we will discuss asymptotic properties of the local linear smooth backfitting estimator. 
	For simplicity we consider only the case that $\mathcal X$ is a product of intervals $\mathcal X_j= (a_j, b_j) \subset \mathbb R$.

	We make the following additional assumptions:
	\begin{enumerate}
	
	\item[(A2)] The observations  $(Y_i, X_i)$ are i.i.d. and the covariates $X_i$ have one-dimansional marginal densities $p_j$ which are strictly positive on $[a_k, b_k]$. The two-dimensional marginal densities $p_{jk}$ of $(X_{i,j}, X_{i,k})$ are continuous on their support $[a_j, b_j] \times [a_k, b_k] $.
		
		\item[(A3)] It holds 
		\begin{equation} \label{eq:model}Y_i= m_0 + m_1(X_{i1}) + \dots + m_d(X_{id})+ \varepsilon_i, \end{equation} 
		for twice continuously differentiable functions $m_j : \mathcal X_j \to \mathbb R$ with 
		$\int m_j(x_j)$ $ p_j(x_j) \mathrm d x_j =0$. The error variables $\varepsilon_i$ satisfy
		$\mathbb E[\varepsilon_i | X_i] =0$ and  $$\sup_{x \in \mathcal X } \mathbb E[|\varepsilon_i| ^{5/2} | X_i=x] < \infty.$$
		\item[(A4)]
			There exist constants $c_1,\dots,c_d> 0$ with $n^{1/5} h_j \to c_j$ for $n \to \infty$. To simplify notation we assume that $h_1=\dots=h_d$. In abuse of notation we write $h$ for $h_j$ and $c_h $ for $c_j$.
	\end{enumerate}
	From now on we will write $\hat m^n =(\hat m^n_0, \hat m^n_1, \dots, \hat m^n_{2d})$ for the estimator $\hat m$ to indicate its dependence on the sample size $n$. The following theorem states an asymptotic expansion for the components  $ \hat m^n_1, \dots, \hat m^n_{d}$.
	Later in this section we will state some lemmas which will be used to prove the result. 
	
	\begin{theorem} \label{theo: asymp}
		Make assumptions (A1) -- (A4). Then
			\begin{align*} &\bigg| \hat m^n_j(x_j) - m_j(x_j)- \left  (\beta_j(x_j) -\int \beta_j(u_j) p_j(u_j) \mathrm d u_j\right )  
			- v_j(x_j)\bigg| \\ &= o_P(h^2+\{nh\}^{-1/2}) = o_P(n^{-2/5}),
			\end{align*} 
			holds uniformly over ${1 \leq j \leq d}$  and ${a_j  \leq x_j \leq b_j }$, 
		where $v_j$ is a stochastic variance term
		\[
		v_j(x_j)= \frac { \frac 1 n \sum_{i=1} ^n  h^{-1} k(h^{-1}(X_{ij} -x_j))\varepsilon_i} { \frac 1 n \sum_{i=1} ^n  h^{-1} k(h^{-1}(X_{ij} -x_j))}=O_p(\{nh\}^{-1/2})
		\]	
		and $\beta_j$ is a deterministic bias term
		\[
		\beta_j(x_j)=\frac 1 2 h^2 m_j^{''}(x_j)\ \frac {b_{j,2}(x_j) ^2- b_{j,1}(x_j) b_{j,3}(x_j)}{b_{j,0}(x_j) b_{j,2}(x_j)- b_{j,1}(x_j) ^2}=O(h^2),
		\]
		with
		$b_{j,l}(x_j)= \int_{\mathcal X_j} k(h^{-1} (u_j-x_j))(u_j-x_j)^l h^{-l-1}b_j(u_j)^{-1}\mathrm du_j$ and $ b_{j}(x_j)= \int_{\mathcal X_j} k(h^{-1} (x_j-w_j)) h^{-1}\mathrm dw_j$ for $0 \leq l \leq 2$. 
	\end{theorem}

	The expansion for $\hat m^n_j$ stated in the theorem neither depends on $d$ nor on functions $m_k$ ($k \not = j$). In particular, this shows that the same expansion holds for the local linear estimator $\widetilde m^n_j$ in the oracle model where the functions $m_k$ ($k \not = j$) are known. More precisely, in the oracle model one observes i.i.d. observations $(Y_i^\ast,X_{ij})$ with
	\begin{align}\label{oracle}
	Y^\ast_i  =  m_j(X_{ij}) + \varepsilon_i , \quad Y^\ast_i = Y_i-\sum_{k \neq j} m_k({X_{ik}}),
	\end{align}
	and the  local linear estimator  $\widetilde m^n_j$ is defined as the second component that minimises 
	the criterion
	\begin{eqnarray*} 
		\widetilde S(f_0, f_j, f^{(1)}_j)&=& \sum_{i=1}^n \int_{\mathcal X} \left \{ Y_i^\ast - f_0 -f_j(x_j) - f _j^{(1)}(x_j)  (X_{ij}-x_j)\right \} ^2 \\ && \qquad \times  \kappa_{h}^{X_{ij}}(X_{ij}-x) \mathrm dx_j
	\end{eqnarray*}
	with boundary corrected  kernel
	\[
	k_h^{u}(u-x)=  \frac{ \kappa \left(\frac{u-x}{h}\right)}{\int_{\mathcal X_j}    \kappa \left(\frac{u-v}{h}\right) \mathrm dv  }.
	\]
	We conclude that the local linear smooth backfitting estimator $\hat m_j$	is asymptotically equivalent to the local linear estimator $\widetilde m^n_j $ in the oracle model. We formulate this asymptotic equivalence as a first corollary of Theorem \ref{theo: asymp}. In particular, it implies that the estimators have the same  first order  asymptotic properties.
	\begin{corollary} \label{corr: compare tilde}
		Make assumptions (A1) -- (A4). Then it holds uniformly over ${1 \leq k \leq d}$  and ${a_j  \leq x_j \leq b_j }$ that
		\begin{eqnarray*} &&\bigg| \hat m^n_j(x_j) -  \widetilde m^n_j(x_j) \bigg| = o_P(h^2).
		\end{eqnarray*}
	\end{corollary} 
	For $x_j \in (a_j + 2 h, b_j - 2h)$ the bias term $\beta_j$ simplifies and we have that
	\[
	\beta_j(x_j) =h^2 \frac 1 2 m_j^{''}(x_j)  \int k(v)v^2\mathrm dv.
	\]
	This implies the following corollary of Theorem \ref{theo: asymp}.
	\begin{corollary} \label{corr: asymp}
		Make assumptions (A1) -- (A4). Then it holds uniformly over ${1 \leq k \leq d}$  and ${a_j + 2 h \leq x_j \leq b_j - 2h}$ that
		\begin{eqnarray*} &&\bigg| \hat m^n_j(x_j) -  m_j(x_j) - \frac 1 2 \left (m_j^{''}(x_j) - \int m_j^{''}(u_j) p_j(u_j) \mathrm d u_j\right ) h^2 \ \int k(v)v^2\mathrm dv  \\
			&& \qquad - v_j(x_j)\bigg| = o_P(h^2).
		\end{eqnarray*}
	\end{corollary} 
	Corollary 	\ref{corr: asymp} can be used to derive the asymptotic distribution of $\hat m^n_j(x_j)$ for an $x_j \in (a_j, b_j)$. Under the additional assumption that $\sigma_j^2(u) = \mathbb E[\varepsilon_i^2| X_{ij}=u]$ is continuous in $u=x_j$ we get under (A1) -- (A4) that $n^{2/5} ( \hat m^n_j(x_j) -  m_j(x_j) )$ has an asymptotic normal distibution with mean $c_h \frac 1 2 \left (m_j^{''}(x_j) - \int m_j^{''}(u_j) p_j(u_j) \mathrm d u_j\right ) \int k(v)v^2\mathrm dv$ and variance $c_h^{-1} \sigma_j^2(x_j) p_j^{-1}(x_j)   \int k(v)^2\mathrm dv$. This is equal to the asymptotic limit distribution of the classical local linear estimator in the oracle model in accordance with Corollary \ref{corr: compare tilde}.\newline
	Now, we come to the proof of  Theorem \ref{theo: asymp}.\newline
	
	First, we define the operator $\mathcal S_{n} = (\mathcal S_{n,0} , \mathcal S_{n,1} ,\dots, \mathcal S_{n,2d} ):\mathcal{G}^n\rightarrow\mathcal{G}^n$ with
		\[
		\mathcal G^n= \{(g_0,\dots,g_{2d})| g_0\in \mathbb R, g_l, g_{l'} \in \text{L}_2(p_l)\ \text{with } P_0(g_l) =0\ \text{for } l=1,\dots,d\},
		\] 
		where $\mathcal S_{n,k}$ maps $g=(g_0,\dots,g_{2d})\in\mathcal{G}^n$ to $f_k$ with
		\begin{align*} f_0 & = \mathcal P_0\left( \sum _{1 \leq l \leq 2d} g_l\right)=\mathcal P_0\left( \sum _{d+1 \leq l \leq 2d} g_l\right) \in \mathbb R,\\
		f_k(x_k)&= \mathcal P_k\left( \sum _{0 \leq l \leq 2d, l\not = k} g_l\right) (x), \end{align*}
		for $1 \leq k \leq 2d$.
	With this notation we can rewrite 
	the backfitting equation \eqref{eq:back} as 
	\begin{equation} \label{eq:back2}  \bar m^n(Y)= \hat m^n  +\mathcal S_{n}  \hat m^n, 
	\end{equation} 
	where for $z \in \mathbb R ^n$ we define $\bar m^n_0(z) = \bar z = \frac 1 n \sum_{i=1} ^n z_i$ and for $1 \leq j \leq d,$
	\begin{eqnarray*} 
		\bar m^n_j(z) (x_j) &=& \hat p_j(x_j) ^{-1} \frac 1 n \sum_{i=1} ^n (z_i - \bar z) \int K_h^{X_i}(X_i -x) \mathrm d x_{-j}, \\\bar m^n_{j'}(z) (x_j) &=& \hat p^{**}_j(x_j) ^{-1} \frac 1 n \sum_{i=1} ^n (X_{ij} -x_j) z_i  \int K_h^{X_i}(X_i -x) \mathrm d x_{-j}.
	\end{eqnarray*} 
	The following lemma shows that $I+\mathcal S_{n}$ is invertible on the event $\mathcal E$. Here we denote the identity operator by $I$.
	
	\begin{lemma} \label{lem:invert}
		On the event $\mathcal E$ the operator $I+ \mathcal S_{n}: \mathcal G^n \to \mathcal G^n$ is invertible. 
	\end{lemma} 
	\begin{proof}
		Suppose that for some $g \in \mathcal G^n $ it holds that $(I+\mathcal S_{n})(g)=0$. We have to show that this implies $g=0$.\newline
		For the proof of this claim note that $g_k + \mathcal S_{n,k}  (g) $ is the orthogonal projection of $\sum_{j=0} ^{2d} g_j$ onto $\mathcal H_k$. Furthermore, we have that $g_k$ is an element of $\mathcal H_k$. This gives that 
		$$\left\langle g_k , \sum_{j=0} ^{2d} g_j\right\rangle_n =0.$$ Summing over $k$ gives
		$$\left\langle  \sum_{j=0} ^{2d} g_j , \sum_{j=0} ^{2d} g_j\right\rangle_n =0.$$
		According to Lemma \ref{lem:uni Hadd} on the event $\mathcal E$ we have $g_0 = 0$ and $g_j \equiv 0$ for $j=1,\dots,2d$. This concludes the proof of the lemma.
	\end{proof}
	
	One can show that under conditions (A1) -- (A4)  the probability of the event $\mathcal E$  converges to one. Note that we have assumed that $\mathcal X=\prod_{j=1} ^d \mathcal X_j$. We conclude that under (A1) -- (A4) $I+\mathcal S_{n}$ is invertible with probability tending to one.\newline
	Thus we have that with probability tending to one
	\begin{eqnarray} \label{eq:back3}  &&\hat m^n - m -\bar m^n(\varepsilon)- \beta_n + \Delta_n m  + \Delta_n \beta_n \\ \nonumber
	&& \qquad= ( I+ \mathcal S_{n})^{-1}   ( I+ \mathcal S_{n}) (\hat  m^n - m - \bar m^n(\varepsilon)-\beta_n + \Delta_n m  + \Delta_n \beta_n),\end{eqnarray} 
	where 
	$m$ has components $m_0,\dots,m_{2d}$ with $m_0,\dots, m_d$ as in \eqref{eq:model} and with $m_{j'} = m_j^{\prime}$ for $1 \leq j \leq d$. 
	Furthermore, $\beta(x)$ has components $\beta_0=0$, $\beta_j(x_j)$ and
	\begin{align*}
	\beta_{j'}(x_j)&=\frac 1 2 m_j^{''}(x_j)\ \frac {b_{j,0}(x_j) b_{j,3}(x_j)- b_{j,1}(x_j) b_{j,2}(x_j)}{b_{j,0}(x_j) b_{j,2}(x_j)- b_{j,1}(x_j) ^2} h
	\end{align*}
	for $j=1,\dots,d$ with $ b_{j,l}(x_j)$ defined above. 
	Additionally, the norming constants are given by
	\begin{align*}
	&(\Delta_n \beta)_j = \int \beta_j (x_j) \hat p_j(x_j) \mathrm d x_j,\\
	&(\Delta_n m)_j = \int m_j (x_j) \hat p_j(x_j) \mathrm d x_j, \\
	&(\Delta_n \beta)_{j'} = (\Delta_n m)_{j'}=0
	\quad \text{for}\ j=1,\dots,d,\\
	&(\Delta_n \beta)_0 =  \sum_{j=1} ^{d}  \int \beta_{j'} (x_j) \hat p^*_j(x_j) \mathrm d x_j,\\
	&(\Delta_n m)_0 =  \sum_{j=1} ^{d}  \int m_{j'} (x_j) \hat p^*_j(x_j) \mathrm d x_j.
	\end{align*}
	One can verify that  for $a_j+ 2 h_j \leq x_j \leq b_j - 2 h_j$ one has $\beta_{j'}(x_j)=o_P(h)$. We have already seen that $\beta_j(x_j) = \frac 1 2 m_j^{''}(x_j)  \int k(v)v^2\mathrm dv + o_P(h^2)$ holds for such $x_j$.\newline
	For the statement of Theorem \ref{theo: asymp} we have to show that for $1 \leq j \leq d$ the $j$-th component on the left hand side of equation \eqref{eq:back3} is of order $o_P(h^2)$ uniformly for $a_j +2h \leq x_j \leq b_j - 2h$.\newline
	For a proof of this claim we first analyze the term 
	\begin{align} \label{eq:back4}  D_n&=( I+ \mathcal S_{n}) (\hat  m^n - m -\bar m^n(\varepsilon) - \beta_n+ \Delta_n m  + \Delta_n \beta_n) \\ \nonumber
	& =  \bar m^n(Y) - ( I+ \mathcal S_{n}) (m +\bar m^n(\varepsilon) + \beta_n - \Delta_n m  - \Delta_n \beta_n).\end{align}
	For this sake we split the term $ \bar m^n(Y)$ into the sum of a stochastic variance term and a deterministic expectation term:
	\begin{align}\label{error_decomposition}
	\bar m^n(Y)  = \bar m^n(\varepsilon) + \sum_{j=0}^d\bar m^n(\mu_{n,j}),
	\end{align}
	where
	\begin{align*}
	\varepsilon &= Y-\sum_{j=0}^d \mu_{n,j},\\
	\mu_{n,j} &=(m_j(X_{ij}))_{i=1,..,n}  \quad \text{for}\ j=1,\dots,d,\\
	\mu_{n,0} &=(m_0)_{i=1,..,n} .
	\end{align*}
	\newline
	We write $D_n=D_n^{\beta}+D_n^\varepsilon$, with
	\begin{align*}
	D_n^{\beta}&= \sum_{j=0}^d\bar m^n(\mu_{n,j}) - ( I+ \mathcal S_{n}) (m + \beta_n - \Delta_n m  - \Delta_n \beta_n) , \\
	D_n^{\varepsilon}&= \mathcal S_{n} (\bar m^n(\varepsilon)).
	\end{align*}
	The following lemma treats the conditional expectation term  $D^{\beta}$.
	
	\begin{lemma}\label{lemma:bias} Assume (A1) -- (A4).
		It holds 	
		$D^\beta_{n,0} = o_p(h^2) $ and 
		\[
		\sup_{x_k\in \mathcal X_k}\big|D^\beta_{n,k}(x_k)\big|=
		\begin{cases} o_p(h^2) &\text{for } \ 1\leq k\leq d, \\
		o_p(h) &\text{for } \ d+1\leq k\leq 2d .
		\end{cases}
		\]
	\end{lemma}
	\begin{proof}
		The lemma follows by application of  lengthy calculations using second order Taylor expansions 
		for $m_j(X_{ij})$ and by application of laws of large numbers.
	\end{proof}
	We now turn to the variance term.

	\begin{lemma}\label{lemma:var} Assume (A1) -- (A4).
		It holds $D^\varepsilon_{n,0} = o_p(h^2) $ and 
		\[
		\sup_{x_k\in \mathcal X_k}\big|D^\varepsilon_{n,k}(x_k)\big|=
		\begin{cases} o_p(h^2) &\text{for } \ 1\leq k\leq d, \\
		o_p(h) &\text{for } \ d+1\leq k\leq 2d .
		\end{cases}
		\]
	\end{lemma}
	\begin{proof}
		One can easily check that  $D^\varepsilon_{n,k}(x_k)$ consists of weighted sums of $\varepsilon_i$ where the weights are of the same order for all $1 \leq i \leq n$. For fixed $x_k$ the sums are of order $O_P(n^{-1/2})$ for $1 \leq k \leq d$ and of order $O_P(h^{-1} n^{-1/2})$ for $d+1 \leq k \leq 2d$. Using the conditional moment conditions on $\varepsilon_i$ in Assumption (A3) we get the uniform rates stated in the lemma.
	\end{proof}
	
	It remains to study the behaviour of  $( I+ \mathcal S_{n})^{-1} D_n^\varepsilon$ and $( I+ \mathcal S_{n})^{-1} D_n^\beta$. We will use a small transformation of $\mathcal S_{n}$ here which is better suitable for an inversion. Define the following $2 \times 2$ matrix $A_{n,k}(x)$ by
	$$A_{n,k}(x) =\frac 1 {\hat p_k \hat p_k^{**} - (\hat p_k^{*})^2} \left(\begin{array}{cc}\hat p_k^{**}\hat p_k & \hat p_k^{**}\hat p_k^{*} \\\hat p_k^{*}\hat p_k & \hat p_k\hat p_k^{**}\end{array}\right)(x_k).$$
	Furthermore, define the $2d \times 2d$ matrix $A_n(x)$ where the elements with indices $(k,k), (k,k'),$  $ (k', k),(k',k')$ are equal to the elements of $A_{n,k}(x)$ with indices $(1,1), (1,2), (2,1),(2,2)$. We now define $\tilde {\mathcal S}_{n}$ by the equation $I+ \tilde {\mathcal S}_{n} = A_n(I+ \mathcal S_{n})$. Below we will make use of the fact that $\tilde {\mathcal S}_{n}$ is of the form 
	\begin{eqnarray} \label{eq:AQ1} \tilde {\mathcal S}_{n,k}m(x) = \sum_{l \not \in \{k,k'\}} \int q_{k,l} (x_k,u) m_l(u) \mathrm d u+ \sum_{l  \in \{k,k'\}} \int q_{l} (u) m_l(u) \mathrm d u,\\
	\label{eq:AQ2} \tilde {\mathcal S}_{n,k'}m(x) = \sum_{l \not \in \{k,k'\}} \int q_{k',l} (x_k,u) m_l(u) \mathrm d u+ \sum_{l  \in \{k,k'\}} \int q_{l} (u) m_l(u) \mathrm d u\end{eqnarray}
	for $1 \leq k \leq d$ with some random functions $q_{k,l},q_{l}$ which fulfill  that $\int  q_{k,l}(x_k, u ) ^2 \mathrm d u$ and  $\int  q_k( u ) ^2 \mathrm d u$ are of order $O_P(1)$ uniformly over $1 \leq k,l\leq 2d$ and $x_k$.

	Note that we need   $\tilde {\mathcal S}_{n}$ because ${\mathcal S}_{n}$ can not be written in the form of \eqref{eq:AQ1} and \eqref{eq:AQ2}. The operator $\tilde {\mathcal S}_{n}$ differs from ${\mathcal S}_{n}$ 
	in the $h$-neighbourhood of the boundary by terms of order $h^2$. Otherwise the difference is  of order $o_p(h^2)$.
	Outside of the $h$-neighbourhood of the boundary, for $n \to \infty$,  the matrix $A_n(x)$ converges to the identity matrix.
	Thus   $\tilde {\mathcal S}_{n}$  is a second order modification of ${\mathcal S}_{n}$  with the advantage of having \eqref{eq:AQ1}-\eqref{eq:AQ2}.\newline
	For our further discussion we now introduce the space $\mathcal G^0$ of tuples $f=(f_0, f_1,\dots,$ $f_{2d})$ with $f_0=0$ and $f_k,f_{k'} : \mathcal X_k \to \mathbb R$ with $\int f_k(x_k) p_k(x_k) \mathrm d x_k =0$ and endow it with the norm $\|f\|^2 = \sum _{k=1} ^{d} (f_k(x_k)^2  + f_{k'}(x_k)^2) p_k(x_k) \mathrm d x_k $. 
	The next lemma shows that the norm of  $H_n( I+ \mathcal S_{n})^{-1} D_n^\varepsilon$ and $H_n( I+ \mathcal S_{n})^{-1} D_n^\beta$ is of order $o_P(h^2)$. Here $H_n$ is a diagonal matrix where the first $d+1$ diagonal elements equal 1. The remaining elements are equal to $h$.
	
	\begin{lemma}\label{lemma:rate of norm} Assume (A1) -- (A4).
		Then it holds that
		$\|H_n( I+ \mathcal S_{n})^{-1} D_n^*\| = \|H_n( I+\tilde{ \mathcal Q}_{n})^{-1} A_nD_n^*\|=o_P(h^2)$ for $D_n^*=D_n^\varepsilon$ and $D_n^*= D_n^\beta$.\end{lemma}
	\begin{proof}
		Define  $\bar D_n^\varepsilon$ and $\bar D_n^\beta$ by $\bar D_{n_k}^\varepsilon(x_k) = D_{n,k}^\varepsilon(x_k) - \int D_{n,k}^\varepsilon(u_k) p_k(u_k) \mathrm d u_k$ and $\bar D_{n,k}^\beta(x_k) = D_{n,k}^\beta(x_k) - \int  D_{n,k}^\beta(u_k) p_k(u_k) \mathrm d u_k$ for $1 \leq k \leq d$ and $\bar D_{n,k}^\varepsilon= D_{n,k}^\varepsilon$ and $\bar D_{n,k}^\beta=D_{n,k}^\beta$, otherwise. 		It can be checked that it suffices to prove the lemma with $ D_n^\varepsilon$ and $D_n^\beta$ replaced by $\bar D_n^\varepsilon$ and $\bar D_n^\beta$. Note that $\bar D_n^\varepsilon$ and $\bar D_n^\beta$ are elements of $\mathcal G^0$. For the proof of this claim we compare the operator $\tilde{ \mathcal S}_{n}$ with the operator
		$\mathcal S_{0}$ defined by $\mathcal S_{0,0}(g) =0$, $\mathcal S_{0,k'}(g)(x_k) = 0$ and $$\mathcal S_{0,k}(g)(x_k) = \sum_{j\not = k} \int_{\mathcal X_j} g_j(u_j) \frac {p_{j,k} (u_j,x_k)}{p_{k} (x_k)} \mathrm d u_j$$
		for $1 \leq k \leq d$.
		By standard kernel smoothing theory one can show that 
		$$ \sup _{g\in \mathcal G^0, \|g \| \leq 1} \| (\mathcal S_{0} - H_n \tilde{ \mathcal S}_{n} H_n^{-1}) g \| = o_P(1).$$
		For the proof of this claim one makes use of the fact that non-vanishing differences in the $h$-neighbourhood of the boundary are 	asymptotically 	 negligible in the calculation of the norm  because the size of the neighbourhood converges to zero.\newline
		In the next lemma we will show that $I + \mathcal S_{0}$ has a bounded inverse. This implies the statement of the lemma by applying the following expansion:
		\begin{eqnarray*}
			&&(I + H_n \tilde{ \mathcal S}_{n} H_n^{-1})^{-1}- (I + \mathcal S_{0})^{-1} \\
			&& \qquad = (I + \mathcal S_{0})^{-1} ((I + \mathcal S_{0}) (I + H_n \tilde{ \mathcal S}_{n} H_n^{-1})^{-1} -I)\\
			&& \qquad = (I + \mathcal S_{0})^{-1}(((I + H_n \tilde{ \mathcal S}_{n} H_n^{-1}) (I + \mathcal S_{0})^{-1} )^{-1} -I)\\
			&& \qquad = (I + \mathcal S_{0})^{-1}  ( (I + (H_n \tilde{ \mathcal S}_{n} H_n^{-1}-  \mathcal S_{0}) (I + \mathcal S_{0})^{-1} )^{-1} -I)
			\\
			&& \qquad = (I + \mathcal S_{0})^{-1}  \sum_{j=1} ^\infty (I + (-1)^j(H_n \tilde{ \mathcal S}_{n} H_n^{-1}-  \mathcal S_{0}) (I + \mathcal S_{0})^{-1} )^j.
		\end{eqnarray*}
		This shows the lemma because of $H_n( I+ \tilde{ \mathcal S}_{n})^{-1}A_nD_n^* = H_n( I+ \tilde{ \mathcal Q}_{n})^{-1}H_n^{-1}H_n A_nD_n^* = (I+H_n\tilde{ \mathcal S}_{n} H_n^{-1} )^{-1}H_n A_nD_n^* $ for $D_n^*=D_n^\varepsilon$ and $D_n^*=D_n^\beta$. \end{proof}
	
	\begin{lemma}\label{lemma:bd inverse} Assume (A1) -- (A4).
		The operator $I+ \mathcal S_{0}: \mathcal G^0 \to \mathcal G^0$ is bijective and has a bounded inverse.
	\end{lemma}
	\begin{proof}
		For a proof of this claim it suffices to show that the operator $I+ \mathcal S_{*}: \mathcal G^* \to \mathcal G^*$ is bijective and has a bounded inverse where  $\mathcal G^*$ is the space of tuples $f=(f_1,\dots,f_{d})$ where  $f_k : \mathcal X_j \to \mathbb R$ with $\int f_k(x_k) \mathrm d x_k =0$ with  norm $\|f\|^2 = \sum _{k=1} ^{d} f_k(x_k)^2  p_k(x_k) \mathrm d x_k $ and $$\mathcal S_{*,k}(g)(x_k) = \sum_{j\not = k} \int_{\mathcal X_j} g_j(u_j) \frac {p_{j,k} (u_j,x_k)}{p_{k} (x_k)} \mathrm d u_j$$
		for $1 \leq k \leq d$. We will apply the bounded inverse theorem. For an application of this theorem we have to show 
		that $I+ \mathcal S_{*}$ is bounded and bijective. It can easily be seen that the operator is bounded. It remains to show that it is 
		surjective. We will show that 
		
		(i) $(I+ \mathcal S_{*}) g^n \to 0$ for a sequence $g^n \in \mathcal G^*$ implies that $g^n \to 0$.
		
		(ii) $ \int g_k (I+ \mathcal S_{*})_k r (x_k) p_k(x_k) \mathrm d x_k =0$ for all $ g \in  \mathcal G^*$ implies that $r=0$.
		
		\noindent Note that (i) implies that $ \mathcal G^{**} = \{ (I+ \mathcal S_{*}) g: g \in  \mathcal G^{*}\}$ is a closed subset of $ \mathcal G^{*} $. To see this suppose that $(I+ \mathcal S_{*}) g^n \to g$ for $g,g^n \in \mathcal G^{*} $. Then (i) implies that $ g^n$ is a Cauchy sequence and thus $ g^n$ has a limit in $\mathcal G^{*} $ which implies that $(I+ \mathcal S_{*}) g^n$ has a limit in $ \mathcal G^{**}$. Thus $ \mathcal G^{**} $ is closed.\newline
		From (ii) we conclude that the orthogonal complement of $ \mathcal G^{**}$ is equal to $\{0\}$. Thus the closure of $ \mathcal G^{**}$ is equal to $ \mathcal G^{*}$. This shows that $ \mathcal G^{*}= \mathcal G^{**}$ because $\mathcal G^{**}$ is closed. We conclude that $(I+ \mathcal S_{*})$ is surjective.\newline
		It remains to show (i) and (ii). Fo a proof of (i) note that $(I+ \mathcal S_{*}) g^n \to 0$ implies that 
		$$\int g_k^n(x_k) (I+ \mathcal S_{*})_k g^n(x_k) p_k(x_k) \mathrm d x_k \to 0$$
		which shows $$\sum_{k=1}^d \int g_k^n(x_k) ^2 p_k(x_k) \mathrm d x_k + \sum_{k\not = j} g_k^n(x_k) p_{kj}(x_k,x_j) g_j(x_j) \mathrm d x_k \mathrm d x_j \to 0.$$
		Thus we have $$\mathbb E[ (\sum _{k=1} ^d g_k(X_{ik}))^2] \to 0.$$
		By application of Proposition \ref{lprop:Hadd} (ii) we get that $\max_{ 1 \leq k \leq d} \mathbb E[  g_k(X_{ik})^2] \to 0$, which shows (i).\newline
		Claim (ii) can be seen by a similar argument. Note that $ \int g_k (I+ \mathcal S_{*})_k r p_k(x_k) \mathrm d x_k =0$ for all $ g \in  \mathcal G^*$ implies that $ \int r_k (I+ \mathcal S_{*})_k r (x_k)p_k(x_k) \mathrm d x_k =0$.
	\end{proof}
	We now  apply the results stated in the lemma for the final proof of Theorem \ref{theo: asymp}.
	
	\begin{proof}[of Theorem \ref{theo: asymp}]
		From \eqref{eq:back3} and Lemma \ref{lemma:rate of norm} we know that the L$_2$ norm of $\hat m^n - m -\bar m^n(\varepsilon)- \beta_n + \Delta_n m  + \Delta_n \beta_n=H_n( I+ \mathcal S_{n})^{-1} (D_n^\varepsilon + D_n^\beta) = H_n( I+\tilde{ \mathcal S}_{n})^{-1} A_n(D_n^\varepsilon + D_n^\beta)$ is of order $o_P(h^2)$. Note that $H_n( I+\tilde{ \mathcal S}_{n})^{-1} = H_n - H_n \tilde{ \mathcal S}_{n} ( I+\tilde{ \mathcal S}_{n})^{-1}. $ We already know that the sup norm of all components in $H_nA_n(D_n^\varepsilon + D_n^\beta)$ are of order $o_P(h^2)$. Thus, it remains
		to check that the sup norm of the components of $H_n \tilde{ \mathcal S}_{n} ( I+\tilde{ \mathcal S}_{n})^{-1}A_n(D_n^\varepsilon + D_n^\beta)$ is of order $o_P(h^2)$. But this follows by application of the just mentioned bound on the L$_2$ norm of $H_n( I+ \mathcal S_{n})^{-1} (D_n^\varepsilon + D_n^\beta)$, by equations \eqref{eq:AQ1} --
		\eqref{eq:AQ2}, and the bounds for the random functions $q_{k,l}$ and $q_{l}$ mentioned after the statement of the equations. One gets a bound for the sup norms by application of the Cauchy Schwarz inequality. 
	\end{proof}
	
	\appendix
	\section{Projection operators} \label{sec: projection operators}
	In this section we will state expressions for the  projection operators
	$\mathcal  P_0$, $ \mathcal  P_k$, $P_k$ and $\mathcal  P_{k'}$ ($1 \leq k \leq d$) mapping elements of $ \mathcal H$ to $\mathcal H_0$, $\mathcal H_k$, $\mathcal H_k +  \mathcal H_0 $ and $\mathcal H_{k'}  $, respectively, see Section \ref{sec:local linear}. For an element  $f = (f^{i,j})_ {i=1,\dots,n;\ j=0,\dots,d}$ the operators $ \mathcal  P_0$, $ \mathcal  P_k$, and $ P_k$ ($1 \leq k \leq d$) set all components to zero but  the components with indices $(i,0), i=1,\dots,n$. Furthermore,  in the case $d < k \leq 2d$ only  the components with index $(i,k-d), i=1,\dots,n$ are non-zero. Thus, for the definition of the operators it remains to set
	\begin{eqnarray*} 
		(\mathcal P_0(f) ) ^{i,0}(x)&=& \frac 1 n {\sum_{i=1}^n\int_{\mathcal X}\{f^{i,0}(u)+\sum_{j=1}^d f^{i,j}(x)(X_{ij}-u_j)\}
			K^{X_i}_h(X_i-u)\mathrm du}.
	\end{eqnarray*} 
	For $1 \leq k \leq d$ it suffices to define $(\mathcal P_k(f) ) ^{i,0}(x) = (P_k(f) ) ^{i,0}(x)- (\mathcal P_0(f) ) ^{i,0}$ and
	\begin{eqnarray*} 
		&& (P_k(f) ) ^{i,0}(x)
		=\frac{1}
		{\hat p_k(x_k)}  \bigg [\frac 1 n \sum_{i=1}^n\int_{u \in \mathcal X_{-k}(x_k) } \bigg \{f^{i,0}(u)+\sum_{j=1}^d f^{i,j}(u)(X_{ij}-u_j)\bigg\}  		\\ 
		&&\qquad \times K^{X_i}_h(X_i-u)\mathrm du_{-k} \bigg],\end{eqnarray*} 
	\begin{eqnarray*} && (P_{k'}(f) ) ^{i,0}(x)
		=\frac{1}
		{\hat p^{**}_k(x_k)}  \bigg [\frac 1 n \sum_{i=1}^n\int_{u \in \mathcal X_{-k}(x_k) } \bigg \{f^{i,0}(u)+\sum_{j=1}^d f^{i,j}(u)(X_{ij}-u_j) 			\bigg\}  \\ 
		&&\qquad \times (X_{ik}-x_k) K^{X_i}_h(X_i-u)\mathrm du_{-k} \bigg].
	\end{eqnarray*} 
	For the orthogonal projections of functions $m \in \mathcal H_{add} $ one can use simplified formulas. In particular,  these formulas can be used in our algorithm for updating functions $m \in \mathcal H_{add} $. If  $m \in \mathcal H_{add} $ has components $m_0,\dots,m_d, m^{(1)}_1,\dots,m^{(1)}_d$ the operators $P_k $ and $P_{k'}$ are defined as follows
	\begin{eqnarray*} 
		&&(\mathcal P_0(m)(x))) ^{i,0} = m_0 + \sum_{j=1}^d \int _{\mathcal X_j} m_j^{(1)}(u_j) \widehat p_j^*(u_j) \mathrm d u_j, \\ 
		&&(P_k(m)(x)) ^{i,0}
		=m_0 +
		m_k(x_k) + 
		m_k^{(1)}(x_k)\frac{ \hat p^{*}_k(x_k) }{\hat p_k(x_k) }
		\\ && \qquad  
		+\sum_{1 \leq j \leq d, j \not = k} \int_{\mathcal X_{-k,j}(x_k) } 
		\bigg [m_j(u_j)   \frac{ \hat p_{jk}(u_j,x_k)}
		{\hat p_k(x_k) } 
		+ m_j^{(1)}(u_j) \frac{ \hat p^{*}_{jk}(u_j,x_k)}
		{\hat p_k(x_k) } \bigg]\mathrm du_{j},\\
		&&(\mathcal P_k(m)(x)) ^{i,0}
		=
		m_k(x_k) + 
		m_k^{(1)}(x_k)\frac{ \hat p^{*}_k(x_k) }{\hat p_k(x_k) } - \sum_{1 \leq j \leq d} \int_{\mathcal X_{j} } 
		m_j^{(1)}(u_j)  \hat p^{*}_{j}(u_j)
		\mathrm du_{j} \\ 
		&& \qquad
		+\sum_{1 \leq j \leq d, j \not = k} \int_{\mathcal X_{-k,j}(x_k) } 
		\bigg [m_j(u_j)   \frac{ \hat p_{jk}(u_j,x_k)}
		{\hat p_k(x_k) } +
		m_j^{(1)}(u_j) \frac{ \hat p^{*}_{jk}(u_j,x_k)}
		{\hat p_k(x_k) }\bigg ] \mathrm du_{j},\\
		&&(P_{k'}(m)(x)) ^{i,0}
		=
		m^{(1)}_k(x_k) + 
		(m_0+ m_k(x_k))\frac{ \hat p^{*}_k(x_k) }{\hat p^{**}_k(x_k) } \\ 
		&& \qquad
		+\sum_{1 \leq j \leq d, j \not = k} \int_{\mathcal X_{-k,j}(x_k) } 
		\bigg[m_j(u_j)   \frac{ \hat p^{*}_{kj}(x_k, u_j)}
		{\hat p^{**}_k(x_k) } +
		m_j^{(1)}(u_j) \frac{ \hat p^{**}_{jk}(u_j,x_k)}
		{\hat p^{**}_k(x_k) }\bigg] \mathrm du_{j} ,
	\end{eqnarray*} 
	where for $1\leq j,k \leq d$ with $k \not =j$
	\begin{eqnarray*} 
		\hat p_{jk}(x_j,x_k) &=& \frac  1 n \sum_{i=1} ^n \int _{\mathcal X_{-(jk)}(x_j,x_k) }  K^{X_i}_h(X_i-x)\mathrm dx_{-(jk)},\\
		\hat p^{*}_{jk}(x_j,x_k) &=& \frac  1 n \sum_{i=1} ^n \int _{\mathcal X_{-(jk)}(x_j,x_k) } (X_{ij}-u_j) K^{X_i}_h(X_i-x)\mathrm dx_{-(jk)},\\
		\hat p^{**}_{jk}(x_j,x_k) &=& \frac  1 n \sum_{i=1} ^n \int _{\mathcal X_{-(jk)}(x_j,x_k) }(X_{ij}-u_j)(X_{ik}-x_k)  K^{X_i}_h(X_i-x)\mathrm dx_{-(jk)}
	\end{eqnarray*} 
	with $\mathcal X_{-(jk)}(x_j,x_k) =\{ u \in \mathcal X : u_k=x_k, u_j=x_j\}$ and $\mathcal X_{-k,j}(x_k)=\{ u\in \mathcal X _j:$ there exists $v \in \mathcal X$ with $v_k=x_k$ and $v_j = u\}$ and  $u_{-(jk)} $ denoting the vector $(u_l: l \in \{1,\dots,d\}\backslash \{j,k\} )$.\newline

	\section{Proofs of Propositions \ref{lprop:Hadd} and \ref{lprop:Hadd closed help1}} \label{app:B}
	
	In this section we will give proofs for Propositions \ref{lprop:Hadd} and \ref{lprop:Hadd closed help1}. They were used in Section \ref{sec:exist} for the discussion of the existence of the smooth backfitting estimator as well as the convergence of an algorithm for its calculation. 
	
	\begin{proof}[of Proposition \ref{lprop:Hadd}]

		$\mathbf{{(ii) \Rightarrow (i)}}.$
		Let $g^{(n)}\in L$  be a Cauchy sequence. We must show $\lim_{n \rightarrow \infty} g^{(n)} \in L$.
		By definition of $L$ there exist sequences $g_1^{(n)}\in L_1$ and $g_2^{(n)}\in L_2$ such that $g^{(n)}=g_1^{(n)}+g_2^{(n)}$. With \eqref{cond2:closed}, for $i=1,2$ we obtain
		\[
		\norm{g_i^{(n)}-g_i^{(m)}}\leq \frac{1}{c}\norm{g^{(n)}-g^{(m)}} \rightarrow  0.
		\]
		Hence, $g_1^{(n)}$ and $g_2^{(n)}$ are Cauchy sequences. Since $L_1$ and $L_2$ are closed their limits are elements of $L_1\subseteq L$ and $L_2\subseteq L$, respectively. Thus,
		\[
		\lim_{n \rightarrow \infty} g^{(n)}= \lim_{n \rightarrow \infty} g_1^{(n)}+g_2^{(n)} \in L.
		\]
		$\mathbf{{(i) \Rightarrow (iii)}}.$
		We write  $\Pi_1(L_2)=\Pi_1$. Since $L$ is closed, it is a Banach space. Using the closed graph theorem, it suffices to show the following: If $g^{(n)}\in L$ and $\Pi_1 g^{(n)}\in L_1$ are converging sequences with limits $g, g_1$, then $\Pi_1g=g_1$.\newline
		
		\noindent Let $g^{(n)}\in L$ and $\Pi_1 g^{(n)}\in L_1$ be sequences with limits $g$ and $g_1$, respectively.
		Write $g^{(n)}=g_1^{(n)}+g_2^{(n)}$.
		Since
		\[
		\norm{g_2^{(n)}-g_2^{(m)}} \leq \norm{g_1^{(n)}-g_1^{(m)}}+\norm{g^{(n)}-g^{(m)}}
		\]
		$g_2^{(n)}$ is a Cauchy sequence converging to a limit $g_2\in L_2$.
		We conclude $g=g_1+g_2$, meaning  $\Pi_1g=g_1$.\newline
		$\mathbf{{(iii) \Rightarrow (ii)}}.$
		If $\Pi_1$ is a bounded operator, then so is  $\Pi_2$,  since $\norm{g_2} \leq \norm{g}+\norm{g_1}$.
		Denote the corresponding operator norms by $C_1$ and $C_2$, respectively. Then
		$$\max\{\norm{g_1},\norm{g_2}\}\leq \max\{C_1,C_2\}\norm{g}$$
		which concludes the proof by choosing $c=\frac{1}{\max\{C_1,C_2\}}$. \newline
		$\mathbf{{(iii) \Leftrightarrow (iv)}}.$
		This follows from
		\[
		\norm{\Pi_1}= \sup_{g\in L} \frac{\norm{g_1}}{\norm{g}}=\sup_{g_1\in L_1, g_2 \in L_2} \frac{\norm{g_1}}{\norm{g_1+g_2}}= \sup_{g_1\in L_1} \frac{\norm{g_1}}{\mathrm{dist}(g_1,L_2)}=\frac{1}{\gamma(L_1,L_2)}.
		\]
	\end{proof}
	
	\begin{lemma}\label{lemma gamma}
		Let $L_1,L_2$ be closed subspaces of a Hilbert space. For $\gamma$ defined as in Proposition \ref{lprop:Hadd} we have
		\begin{align*}
		\gamma(L_1,L_2)^2 = 1-\norm{\mathcal{P}_2\mathcal{P}_1}^2.
		\end{align*}
	\end{lemma}
	
	\begin{proof}
		\begin{align*}
		\gamma(L_1,L_2)^2 & =\inf_{g_1\in L_1,\norm{g_1}=1}\norm{g_1-\mathcal{P}_2g_1}\\
		& =\inf_{g_1\in L_1,\norm{g_1}=1}\langle g_1-\mathcal{P}_2g_1,g_1-\mathcal{P}_2g_1\rangle\\
		& =\inf_{g_1\in L_1,\norm{g_1}=1}\langle g_1,g_1\rangle-\langle\mathcal{P}_2g_1,\mathcal{P}_2g_1\rangle\\
		& =1-\sup_{g_1\in L_1,\norm{g_1}=1}\langle\mathcal{P}_2g_1,\mathcal{P}_2g_1\rangle\\
		& =1-\sup_{g\in L,\norm{g}=1}\langle\mathcal{P}_2\mathcal{P}_1g,\mathcal{P}_2\mathcal{P}_1g\rangle\\
		& =1-\norm{\mathcal{P}_2\mathcal{P}_1}^2.
		\end{align*}
	\end{proof}
	
	\begin{proof} [of Proposition \ref{lprop:Hadd closed help1}] Let $\mathcal P_j$ be the orthogonal projection onto $L_j$. Following Lemma \ref{lemma gamma} we have
		\begin{align*}
		1-\norm{{\mathcal P_2 \mathcal P_1}}^2=\gamma(L_1,L_2)^2.
		\end{align*}
		Using Proposition \ref{lprop:Hadd}, proving $\norm{\mathcal P_2 \mathcal P_1}<1$ implies that $L$ is closed. Observe that   $\norm{\mathcal P_2 \mathcal P_1}\leq1$ because
		for $g \in L$
		\[
		\norm{\mathcal P_j g}^2= \langle\mathcal P_j g,\mathcal P_j g\rangle=\langle g,\mathcal P_j g\rangle \leq \norm{g}\norm{\mathcal P_j g},
		\]
		which yields $\norm{\mathcal P_i}\leq 1$ for $i=1,2$. To show the strict inequality, note that if $\mathcal P_2{_{\big | L_1}}$ is compact, so is $\norm{{\mathcal P_2 \mathcal P_1}}$ since the composition of two operators is compact if at least one is compact.\newline
		Thus, for every $\varepsilon>0$,  $\mathcal P_2 \mathcal P_1$ has at most a finite number of eigenvalues greater than $\varepsilon$. Since $1$ is clearly not an eigenvalue, we conclude  $\norm{{\mathcal P_1 \mathcal P_2}}<1$.
	\end{proof}

%
%
%
	\bibliographystyle{apalike}

	\bibliography{mybib.bib}

\end{document}